\documentclass[a4paper,oneside,article,10 pt]{memoir}
\usepackage[utf8]{inputenc}
\usepackage[english]{babel}
\usepackage[T1]{fontenc}

\setlrmarginsandblock{3.5cm}{2.5cm}{*} 
\setulmarginsandblock{2.5cm}{*}{1} 
\checkandfixthelayout

\usepackage[utf8]{inputenc}
\usepackage[english]{babel}
\usepackage[T1]{fontenc}

\usepackage{amsmath,amssymb,xcolor}
\usepackage{bm}

\usepackage{mathtools}

\usepackage{amsthm}

\newtheorem{thm}{Theorem}
\newtheorem{lemma}[thm]{Lemma}

\newtheorem{prop}[thm]{Proposition}
\newcommand{\cH}{{\mathcal H}}

\newcommand\N{\mathbb{N}}

\newcommand\R{\mathbb{R}}

\newcommand\x{\mathbf{x}}

\newcommand{\bx}{\mathbf{x}}
\newcommand{\1}{\mathbf{1}}

\DeclareMathOperator{\aff}{aff}
\DeclareMathOperator{\myspan}{span}
\DeclareMathOperator{\conv}{conv}
\DeclareMathOperator{\lin}{lin}

\usepackage[hidelinks = TRUE]{hyperref}

\makeatletter

\def\blfootnote{\gdef\@thefnmark{}\@footnotetext}

\makeatother
\begin{document}

 \title{Rotational Crofton Formulae with a Fixed Subspace}

\author{Emil Dare$^{1,*}$\footnote{∗Corresponding author
Email addresses: dare@math.au.dk (Emil Dare ), kiderlen@math.au.dk (Markus Kiderlen)}, Markus Kiderlen$^{1}$  \\
        \small $^{1}$\emph{Department of Mathematics, Ny Munkegade 118, 8000, Aarhus C, Denmark} 
        }

\maketitle
\begin{abstract}
The classical Crofton formula explains how intrinsic volumes of a convex body $K$ in $n$-dimensional Euclidean space can be obtained from  integrating a measurement function at sections of $K$ with invariantly moved affine flats. 
Motivated by stereological applications, we present variants of Crofton's formula, where the flats are constrained to contain a fixed linear subspace $L_0$, but are otherwise invariantly rotated. This main result generalizes a known rotational Crofton formula, which only covers the case $\dim L_0=0$.  The proof combines a suitable  Blaschke--Petkantschin formula with the classical Crofton formula.
We also argue that our main result is best possible, in the sense that one cannot estimate intrinsic volumes of a set, 
based on lower-dimensional sections, 
other than those given by our result. 
Finally, we provide a proof for a well-established variant: an integral relation for  vertical sections.  Our formula is stated for  intrinsic volumes of a given set, complementing the  classical approach for Hausdorff measures. 
\end{abstract}

\chapter{Introduction}
\blfootnote{ *Corresponding author \\
 Email addresses: dare@math.au.dk (Emil Dare), kiderlen@math.au.dk (Markus Kiderlen)}
The classical Crofton formula expresses an intrinsic volume of a convex body $K$ as 
invariant integral of another intrinsic volume of the intersection of $K$ with affine subspaces.
More specifically,  for a set $K$ in the family $\mathcal K^n$ of convex bodies (nonempty, compact, convex subsets) of $\R^n$,  Crofton's intersection formula (\cite[eq.~(4.59)]{SchneiderRoed}) states 
\begin{equation}\label{eqNy1}
    \int_{A(n,k)} \varphi(K\cap E) \, \mu_k (dE)     = V_{m}(K), 
\end{equation}
for $k \in \{0,...,n\}$ and $m\in \{n-k,\ldots,n\}$, where $A(n,k)$ is the family of $k$-dimensional affine subspaces (flats) in $\mathbb R^n$, $\mu_k$ is a motion invariant measure on that space, and $\varphi(K\cap E)$ is proportional to 
$V_{m+k-n}(K\cap E)$.  The functionals $V_j:\mathcal K^n\to \mathbb R$  for 
$j=0,\ldots,n$ appearing in \eqref{eqNy1}, are the \emph{intrinsic volumes} usually defined as polynomial coefficients in \emph{Steiner's formula} (\cite[Thm.~3.10]{LectureNotesConvexGeometry}), comprising 
$V_n(K)$, which is the ordinary volume, $V_{n-1}(K)$, which is proportional to the surface area  and 
 $V_1(K)$ being proportional to the mean width of $K$. We recommend the monograph \cite{SchneiderRoed} as an excellent reference for convex geometric notions and results. 
 
 The measure $\mu_k$ is not finite, but its restriction to the family $A_{K'}$  of flats hitting a compact reference set $K'\supset K$ is. Due to this fact, \eqref{eqNy1} can be used to obtain unbiased estimates of $V_m(K)$ from $K\cap E$, where $E\in A_{K'}$ is invariantly randomized. This has been used extensively in 
 stereology, see \cite{StereologyStat} and the references therein.

 However, in particular in biological applications, it is sometimes more convenient not to randomize over all flats in $A_{K'}$, but only over flats containing a fixed point, usually thought of as the origin. This led to the branch of \emph{local stereology} (see \cite{EvaLocalStereology}) and estimators such as the \emph{nucleator }\cite{Nucleator} and the 
 \emph{rotator} \cite{EstRotator}. Although historically earlier, these estimators are consequences of an underlying integral formula,  the so-called \emph{rotational Crofton formula}, derived independently in \cite{EvaAuneau2010} and \cite{GUALARNAU2010}. It is a variant of the classical Crofton formula and reads
 \begin{equation}\label{eqNy2}
    \int_{G(n,k)} \varphi_L(K\cap L)\,  \nu_k (dL)     = V_{m}(K), 
\end{equation}
for $k \in \{1,...,n-1\}$ and $m\in \{n-k+1,\ldots,n\}$,
where $\nu_k$ is the \emph{rotation} invariant measure on the Grassmannian $G(n,k)$ of $k$-dimensional \emph{linear} subspaces of $\R^n$. The measurement function $\varphi_L(K\cap L)$ is an explicitly known function of $K\cap L$ and $L$, but it is no longer proportional to an intrinsic volume of $K\cap L$.

Some stereological applications require even more constraints on the intersecting planes. For instance, 
in \cite{EstRotator} the volume of a three-dimensional object
is estimated
from sections with ordinary planes, that all 
contain a given line $L_0$. The corresponding estimator is called \emph{vertical rotator}. The practical implementation of this estimator 
for stereological applications can be found in \cite{CruzOrive87b} (not to be confused with \cite{CruzOrive87a}, which gives early historical notes revolving around stereology).
To state the underlying Crofton-type formula in $\R^n$ with a general fixed subspace $L_0\in G(n,r)$, 
$r\in \{0,\ldots,n-1\}$, and for general intrinsic volumes, is the main purpose of this paper. The result reads
\begin{equation}\label{eqNy3}
    \int_{G(L_0,k)} \varphi_L^{L_0}(K\cap L)\, \nu_k^{L_0} (dL)= V_{m}(K), 
\end{equation}
for $k \in \{r+1,...,n\}$ and $m\in \{n+r-k+1,\ldots,n\}$. 
Here $G(L_0,k)\subset G(n,k)$ is the family of all $k$-dimensional linear subspaces containing $L_0$, and $\nu_k^{L_0}$ is the probability measure that is invariant under all rotations fixing $L_0$. 
The new measurement function $\varphi_L^{L_0}(K\cap L)$ is explicitly calculated and given in the main Theorem \ref{thmGeoRotationalCroften} below. 
Clearly \eqref{eqNy3} reduces to the rotational Crofton formula \eqref{eqNy2} if $r=0$, that is, if $L_0=\{o\}$. It is therefore not surprising that the method of proof for \eqref{eqNy3} is a generalization of 
the one for \eqref{eqNy2}. The idea of the latter can best be explained in stereological terms: 
A suitable Blaschke--Petkantschin formula  (\cite[p.~285]{GulBog}) allows generating a random $q$-dimensional flat $E$, $q=\{k-(n-m),\ldots,k-1\}$, in the isotropic subspace $L$ in such a way that the distribution of $E$ is motion invariant in $\R^n$. As $K\cap L\cap E=K\cap E$, the classical Crofton formula can thus be used to 
obtain explicit measurement functions in \eqref{eqNy2}. As $q$ may vary, one obtains $(k-1)-(k-(n-m))+1=n-m$
potentially different measurement functions, which, however, turn out to coincide when $L_0=\{o\}$. 
Our proof of \eqref{eqNy3} proceeds along the same lines, but with a more general Blaschke--Petkantschin formula (Theorem \ref{thm:BP} below). 
We will again see that several potentially different measurement functions can be obtained depending on the parameter $q$. Extending the mentioned uniqueness for $L_0=\{o\}$, 
Theorem \ref{Thm:Uniqueness} states that they all coincide also in the general case. It should be noted that the special case $n=m=3$ of  \eqref{eqNy3} was already derived in  \cite{EstRotator} with a proof based on the Pappus-Guldinus theorem.

Note that \eqref{eqNy3} is different from the well-established stereological  concept of \emph{vertical sections}, where averages are taken over \emph{affine planes} of given dimension that are \emph{parallel} to a fixed subspace $L_0\in G(n,r)$, $r\in \{1,\ldots,n-1\}$. In our notation, the integral formula underlying this concept is 
\begin{equation}\label{eqNy4}
    \int_{G(L_0,k)} \int_{L^\perp}\tilde \varphi_L^{L_0}\big(K\cap (L+x)\big)\,\lambda_{L^\perp}(dx) \,\nu_k^{L_0} (dL)= V_{m}(K), 
\end{equation}
for $k\in \{r+1,\ldots n\}$ and $m\in \{ n+r-k,\ldots n\}$,  where $\lambda_{L^\perp}$ is the Lebesgue-measure on the orthogonal complement $L^\perp$ of $L$. 
   Baddeley \cite{Baddeley83} used this concept to estimate surface area of an object in $\mathbb{R}^3$ based on vertical sections parallel to a fixed line ($r=1$). A practical example involving a Paddington bear can be found in \cite{Baddeley86}.  He extended this idea in 
 \cite{Baddeley84} to arbitrary dimensions $n$ and $r$ for Hausdorff measures of rectifiable sets (\cite[3.2.14]{Federer69}) with a proof based on the coarea formula.  
 For more information about vertical sections in stereology we recommend consulting the book 
 \cite{StereologyStat} or the recent overview on vertical sections in \cite{EvaOverViewStereologi}. To the best of our knowledge, the vertical section formula \eqref{eqNy4} has not been stated for intrinsic volumes in the literature. We therefore give the details of an independent proof based on a Blaschke--Petkantschin-type  result in Appendix A, see Theorem \ref{thmGeoAffineRotationalCroften}.

This paper is organized as follows. In Section 2 we recall preliminaries and  basic notation. In Section 3 we derive the main result, Theorem \ref{thmGeoRotationalCroften}, a rotational  Crofton formula with a fixed subspace. This requires a suitable Blaschke--Petkantschin formula, which will be stated in Theorem \ref{thm:BP}. Section 
\ref {sec:measureFct} collects basic properties and simplified representations  of the measurement functions, including the aforementioned proof of independence from the parameter $q$. Section \ref{sec:limitations} shows that 
no other intrinsic volumes than those 
described in Theorem \ref{thmGeoRotationalCroften} can be obtained 
from section profiles under the given design. Hence, our results cannot be extended using different methods.  The paper is supplied with an appendix discussing the  vertical section formula \eqref{eqNy4}.

\chapter{Notation and preliminaries}
For a set $A \subseteq \mathbb{R}^n$ we let $\myspan(A)$,  $\aff(A)$ 
and $\conv A=\conv(A)$ denote the linear, the affine and the convex hull of $A$, respectively. The set $A^\perp:=\{x\in \R^n: \langle x,y\rangle=0 \text{ for all } y\in A\}$ is the orthogonal complement of $\myspan A$.
We will write $\mathcal{H}^{d}$ for the $d$-dimensional Hausdorff measure on $\mathbb{R}^n$ for $d=0,\ldots,n$, see e.g.~\cite{KonstruktionHausdorffRoger}. We let  $\kappa_n= \cH^n(B^n)= \frac{\pi^{n/2}}{\Gamma(1+n/2)}$ be the volume of the Euclidean unit ball $B^n$  in $\mathbb{R}^n$ and $\omega_n= 
\cH^{n-1}(S^{n-1})=n \kappa_{n}$ be the surface area of the Euclidean unit sphere $S^{n-1}:=\{x \in \mathbb{R}^n \,: \,\|x\|=1 \}$. Throughout the paper we will use the following decomposition 
(\cite[p.~1]{SphericalIntegrationFormula}) of the restriction of $\cH^{n-1}$ to the Borel sets of the unit sphere, which can be thought of as cylindrical coordinates on the sphere: 
 \begin{equation}\label{eqGEOSphericalFormula}
     \int_{S^{n-1}} f(u) \, \mathcal{H}^{n-1}(d u)
     =
     \int_{S^{n-1}\cap v^\perp}\int_{-1}^1 f(tv+\sqrt{1-t^2}w) (1-t^2)^{\frac{n-3}{2}} \, dt \, \mathcal{H}^{n-2}(dw)
 \end{equation}
Here, $v\in S^{n-1}$ is a fixed unit vector and $f:S^{n-1}\to [0,\infty)$ is measurable.

We will follow the notation in \cite[Chpt.~13]{GulBog} for spaces of flats and subspaces in $\mathbb{R}^n$,
$n \in \mathbb{N}$. For $k \in \{0,\ldots n\}$ we let $A(n,k)$ denote the family of \emph{affine} $k$-dimensional flats of $\mathbb{R}^n$ and we let $G(n,k)$ be the \emph{Grassmannian} of $k$-dimensional \emph{linear} subspaces of $\mathbb{R}^n$. As we will work with flats and subspaces incident with others, we define for 
a fixed $L \in G(n,p)$, $p \in \{0,\ldots,n\}$, the spaces
\begin{align*}
    G(L,k):=& \begin{cases}
     \{ L' \in G(n,k) \,: \,L' \subseteq L \}, \text{ if } k \leq p,
     \\
     \{ L' \in G(n,k) \,: \,L \subseteq L' \}, \text{ if } k > p,
    \end{cases}
\end{align*}
and, similarly,  for $E \in A(n,p)$
\begin{align*}
    A(E,k):=& \begin{cases}
     \{ E' \in A(n,k) \,: \,E' \subseteq E \}, \text{ if } k \leq p,
     \\
     \{ E' \in A(n,k) \,: \,E \subseteq E' \}, \text{ if } k > p.
    \end{cases}
\end{align*}
To avoid degenerate situations, we assume throughout this paper that the dimension $n$ of the ambient space is at least 3.
The unique invariant probability measure on $G(L,k)$ will be denoted by $\nu_k^L$. Its construction is outlined in \cite[p.~590]{GulBog}. Invariance is understood here with respect to all rigid rotations that keep $L$ fixed.
Similarly, there is a measure $\mu^E_k$ on $A(E,k)$, which is invariant under all rigid rotations and translations fixing $E$. It is unique up to normalization, which we choose as in \cite{GulBog}. More specifically, if $k\le p$ and 
$E=L+x$ with $L\in G(n,p)$ and $x\in \R^n$, say, then
\begin{align}\label{eqGEOAffineIntegrationToLinearplusShift}
  \int_{A(E,k)} f(E') \,\mu_k^{E}(dE')=
    \int_{G(L,k)}\int_{M^\perp\cap E} f(M+x+y) \,\lambda_{M^\perp\cap E}(dy) \,\nu_k^L (dM)
\end{align}
for any measurable function $f\ge0$; cf.~\cite[eq.~(13.13)]{GulBog}.
Here, and later on,  $\lambda_F$ denotes the $q$-dimensional Lebesgue measure on a flat $F\in A(n,q)$. 
When the flat is clear from the context, we also use the notation $\lambda_q:=\lambda_F$.

In the special case when $p=n$ we notice that $G(L,k)=G(n,k)$ and $A(E,k)=A(n,k)$. We will write $\nu_k$ for the invariant probability measure on $G(n,k)$ and $\mu_k$ for the invariant measure on $A(n,k)$. 
We will use the notation $p(x|E)$ for the orthogonal projection of $x\in \R^n$ on the flat $E\subset \R^n$.

To simplify notation, we use boldface letters to denote vectors of vectors. For instance, we write $\x=(x_1,...,x_q) \in (\mathbb{R}^n)^{q}$ for the $q$-tuple of $n$-dimensional vectors, $q \in \{1,...,n \}$. 

Let   \[
P( \x ) =P(x_1,\ldots,x_q)=\Big\{\sum_{i=1}^q \alpha_i x_i \,: 0\leq \alpha_i\leq 1,\,\,i=1,\ldots,q\Big\}
	 \] denote the parallelepiped spanned by 
	$x_1,\ldots,x_q\in \R^n$. The $q$-dimensional volume of $P(\x)$ is denoted by $\nabla_q(\x )$. 
	
	For $\x \in (\mathbb{R}^n)^q$ the $(q-1)$-dimensional volume $\Delta_{q-1}(\x)$ 
	of  $\conv\{x_1,\ldots,x_q\}$  obeys
\begin{equation}\label{eqGeoDeltaNabla}
    \Delta_{q-1}(x_1,\ldots,x_{q})= \frac{1}{(q-1)!} \nabla_{q-1}(x_2-x_1,\ldots,x_q-x_1),
\end{equation}
 see \cite[eq.~(7.6)]{GulBog}.
 We have
\begin{align}
    \nabla_{q+1}(x_1,\ldots,x_{q+1})=
    \nabla_{q}(x_1,\ldots,x_{q})d(x_{q+1},L),
    \label{eq:Gul7.9}\end{align}
    where $x_{q+1}\in \R^n$ and $L\in G(n,q)$ is a linear space containing $x_1,\ldots,x_{q}$, see \cite[eq.~(7.9)]{GulBog}. 
 Extending these definitions, let 
 \begin{equation}\label{eq:defNablaxM}
     \nabla_{q,r}(\x,M):= \nabla_{q+r}(x_1,\ldots, x_q,u_1, \ldots u_r),
 \end{equation}
 where $(u_1,\ldots ,u_r)$ forms an orthonormal basis of $M \in G(n,r)$. If, in addition, $L \in G(n,k)$ with $r+k \leq n$ is given, the \emph{subspace determinant} of $L$ and $M$ is given as $[L,M]:= \nabla_{k,r}(\textbf{v},M)$, where $\textbf{v}=(v_1,\ldots,v_k)$ forms an orthonormal basis of $L$. In the special cases 
 $M=\{o\}$ or $L=\{o\}$, we define $[M,L]:=1$.
 If $\dim L+\dim M=n$ and $L \cap M = \{o \}$, the subspace determinant 
 is equal to the factor by which the $(\dim M)$-dimensional Lebesgue measure is multiplied under the orthogonal projection from $M$ to $L^\perp$, see \cite[p.~598]{GulBog}. Hence, $\int_{\R^n}f(z)\lambda_n(dz)$ coincides with
\begin{equation}\label{eq0proj}
    \int_{L}\int_{L^\perp} f(x+y)\lambda_{L^\perp}(d x)\lambda_L(dy)%=\int f(z) d\lambda_n(z)
    =[L,M]\int_{L}\int_{M} f(x+y)\lambda_{M}(d x)\lambda_{L}(dy)
	\end{equation}
	 for any measurable function $f:\R^d\to [0,\infty)$. 
	 \medskip 

For $q,r\in \{0,\ldots,n-1\}$ with $q+r\le n-1$ let
\begin{equation}
    \label{eq:DefD}
    D(E,L):=d(o,E)[\myspan E,L], \qquad E\in A(n,q), L\in G(n,r),
\end{equation}
where   $d(o, E)$ is the Euclidean distance of $o$ to $E$. This quantity will play an important role later on, and we remark some of its properties in the following lemma. 

\begin{lemma}\label{lemD(E,L)}
Let $L\in G(n,r)$ and $E\in A(n,q)$ with $q+r\le n-1$ be given. 
\begin{itemize}
    \item[(i)] If $E=\aff\{x_0,\ldots,x_q\}\in A(n,q)$ with $\x=(x_0,\ldots,x_q)\in (\R^n)^q$ then 
\begin{equation}\label{eqGeoNablaSubspace}
D(E,L)=\frac1{q!}\frac{\nabla_{q+1,r}(\bx,L)}{\Delta_q(\bx)}.
\end{equation}
    \item[(ii)] If $E=M+z$, $M\in G(n,q)$, then 
\begin{equation}\label{eq:DefDnew}
D(E,L)=d(E,L)[M,L],
\end{equation}
where $d(E,L)$ is the Euclidean distance between $E$ and $L$. 
\item[(iii)] If $E=M+z, M \in G(n,q), z \in M^\perp$, 
$q+r\leq n-2$ and $u\in S^{n-1} \cap M^\perp$, then 
			\begin{equation}\label{eqOneDimMore}
				D(E+\myspan\{u\},L)=D(E,L) \,
				\big\|p\big(u\big|(L+M+z)^\perp\big)\big\|.
			\end{equation}
\end{itemize} 
\end{lemma}
\begin{proof}
We start with the proof of (i).
The claim is trivial for $r=0$, so we may assume that ${\mathbf v}=(v_1,\ldots,v_r)$ is an orthonormal basis of $L$. Define  $M:=\myspan E=\myspan\{x_0,\ldots,x_q\}$. If $M\cap L\ne \{o\}$, both sides of \eqref{eqGeoNablaSubspace} vanish, so we may also assume 
$M\cap L= \{o\}$, implying $P(x_0,\ldots,x_q,v_1,\ldots,v_r)=P(\bx)+P(\mathbf v)$. 
This, the definition \eqref{eq:defNablaxM}  and relation \eqref{eq0proj} applied in $M+L$ yield \begin{align*}
\nabla_{q+1,r}(\x, L)
&=
 \int_{\R^{q+1+r}} 1_{P(\x)+P(\mathbf v)}(z)\,\lambda_{q+1+r}(d z)
 \\
 &=
    [L,M] 
    \int_{L} \int_{M} 1_{P(\x)+P(\mathbf v)}(y+z) \, \lambda_{M}(dz) \,\lambda_{L}(dy)
   \\& =
    [L,M]\,  \lambda_r(P(\mathbf v))  \lambda_{q+1}(P(\x))
    \\&=
     [L,M]  \nabla_{q+1}(\x).
\end{align*}
Now \eqref{eqGeoDeltaNabla}, the definition of $\Delta_q$ and the calculation of the volume of a $(q+1)$-dimensional pyramid (using Tonelli's theorem) gives
\begin{align*}
    \nabla_{q+1}(\bx)&=(q+1)!\Delta_{q+1}(o,x_1,\ldots,x_q)
    \\&=q!\,\Delta_{q}(x_0,\ldots,x_q)d\big(o,\aff\{x_0,\ldots,x_q\}\big)
    \\
    &=q!\, \Delta_{q}(\bx)d(o,E).
\end{align*}
This shows assertion (i).

To show (ii) let $E=M+z$, $z\in M^\perp$, and assume that $u_1,\ldots,u_q$ is an orthonormal basis of $M$. As the affine hull of $z,z+u_1,\ldots,z+u_q$ is $E$, relation \eqref{eqGeoNablaSubspace} and 
\[
(q!)\Delta_q(z,z+u_1,\ldots,z+u_q)=1
\]
(by \eqref{eqGeoDeltaNabla}) show 
\begin{align*}
D(E,L)&=\nabla_{q+1,r}(z,z+u_1,\ldots,z+u_q,L)
\\&=\nabla_{q+1,r}(z,u_1,\ldots,u_q,L),
\end{align*}
where \eqref{eq:Gul7.9} was applied back and forth $q$ times at the last step. 
Now
\eqref{eq:Gul7.9}
yields 
\begin{align*}
D(E,L)&=d(z,M+L)\nabla_{q,r}(u_1,\ldots,u_q,L)
 =d(E,L)[M,L],
\end{align*}
as claimed.

We prove assertion (iii). 
		If $z\in \myspan\{u\}$, the claim \eqref{eqOneDimMore} is trivially true in view of the definition of $D(E,L)$, so we assume  $z\not \in \myspan\{u\}$ in the following.  Using 
		$z\in M^\perp$ and $u\in S^{n-1}$ we obtain 
		\begin{align}
			d(z,M+\myspan\{u\})^2=\|z-\langle z,u\rangle u\|^2=\|z\|^2\,\|p(u|z^\perp)\|^2. 
			\label{eq:dnew}
		\end{align}
		If $u_1,\ldots u_q$ is an orthonormal basis of $M$ then 
		$u_1,\ldots u_q, \frac{z}{\|z\|}, \frac{p(u|z^\perp)}{\|p(u|z^\perp)\|}$
		is an orthonormal basis of $M+\myspan\{z,u\}$.  The definition of $D(E,L)$ and \eqref{eq:dnew} give
		\begin{align*}
			D(E+\myspan\{u\},L)
			&=d(o,z+M+\myspan\{u\})\,\big[M+\myspan\{z,u\},L\big]
			\\
			&=d(z,M+\myspan\{u\})
			\,\nabla_{q+2,r}\Big(u_1,\ldots u_q, \frac{z}{\|z\|}, \frac{p(u|z^\perp)}{\|p(u|z^\perp)\|},L\Big).  
			\\
			&=\|z\|
			\,\nabla_{q+2,r}\Big(u_1,\ldots u_q, \frac{z}{\|z\|}, p(u|z^\perp),L\Big).  
		\end{align*}
		Using that $(L+M+z)^\perp \subseteq z^\perp$ we conclude from \eqref{eq:Gul7.9},
  applied to the vector $\frac{p(u| z^\perp)}{\|p(u|z^\perp\|}$, that 
		\begin{align*}
			D(E+\myspan\{u\},L)
			&=\|z\|\,\nabla_{q+1,r}\Big(u_1,\ldots u_q, \frac{z}{\|z\|}, L\Big)\,
			\big\|p\big(u\big|(L+M+z)^\perp\big)\big\|.
		\end{align*}
		As $\|z\|=d(o,E)$ and 
			$u_1,\ldots u_q, \frac{z}{\|z\|}$
			is an orthonormal basis of $M+\myspan\{z\}$, we have \\
		$\|z\|\,\nabla_{q+1,r}\Big(u_1,\ldots u_q, \frac{z}{\|z\|}, L\Big)=
		D(E,L) $ and the assertion follows.

\phantom\qedhere
\end{proof}
The first statement in Lemma \ref{lemD(E,L)} is a rather technical relation, only used in the proof of Theorem \ref{thm:BP}. Lemma \ref{lemD(E,L)}.(ii) shows that $D(E,L)$ depends only on the distance between the two flats and the generalized angle between the corresponding linear subspaces. In particular, it shows that if $E \in A(n,0)$ and $L \in G(n,r)$ then $D(E,L)=\| p(E| L^\perp)\|$. The last statement in Lemma \ref{lemD(E,L)} is used in the proof of Theorem \ref{Thm:Uniqueness}. 

\medskip

We define the constants
 \begin{equation}\label{eqGeoDefB}
      b_{n,q}:= \frac{\omega_{n-q+1}\cdots \omega_n}{\omega_1 \cdots \omega_q},  \quad 1 \leq q \leq n,
 \end{equation}
 needed in certain  
 Blaschke--Petkantschin formulae. Fix integers $k,q \geq 1$ and $r \geq 0$ such that
 $q+r \leq k \leq n$ and a subspace $L_0 \in G(n,r)$. Generically, the entries of a vector $\x$ consisting of $q$ points in $\R^n$ and $L_0$ span a subspace $L\in G(L_0,r+q)$, which will be called a \emph{pivot}. The following Blaschke--Petkantschin formula (with $k=q+r$) shows that invariant integration with respect to $\x$ can be achieved by invariantly integrating $\x$ in $L$, and then integrating the pivot:
\begin{align}\label{thmGeo722}
    \int_{(\mathbb{R}^n)^q} f(\x) \,\lambda^q(d \x)
    =
    \frac{b_{n-r,q}}{b_{k-r,q}} \int_{G(L_0,k)}\int_{L^q} f(\x)  \nabla_{q,r}(\x, L_0)^{n-k} \,\lambda_L^{q}(d \x) \,\nu_{k}^{L_0}(dL)
\end{align}
holds for all measurable functions $f: (\mathbb{R}^n)^q\rightarrow [0,\infty)$, see \cite[Thm.~7.2.2]{GulBog}. 
A proof of the  Blaschke--Petkantschin formula  \eqref{thmGeo722}, based on polar decomposition of the $k$-fold product of Lebesgue measure, can be found in \cite{AlternativeBPandAffineBP2,AlternativeBPandAffineBP1,AlternativeBPandAffineBP}. 
A similar, affine  
Blaschke--Petkantschin formula   (\cite[Thm.~7.2.7]{GulBog}) states 
\begin{equation}\label{thmGeo727}
    \int_{(\mathbb{R}^n)^{q+1}} f (\x)  \, \lambda^{q+1}(d\x)
    =
    b_{n,q}(q!)^{n-q} \int_{A(n,q)} \int_{E^{q+1}} f(\x) \Delta_q (\x) ^{n-q}\,\lambda_E^{q+1}(d\x) \,\mu_q (dE)
\end{equation}
for all measurable functions $f: (\mathbb{R}^n)^{q+1}\rightarrow [0,\infty)$ and $q \in \{1,...,n\}$.
An overview of Blaschke--Petkantschin-type formulas with a sphere as pivot can be found in \cite{Nikitenko2019BPSpheres}.

For later reference, we state the classical Crofton formula (see,  e.g.,~\cite[Thm.~5.10]{LectureNotesConvexGeometry} or \cite[eq.~(4.59)]{SchneiderRoed}), which was already mentioned in the introduction in \eqref{eqNy1}: 
\begin{equation}\label{thmGeo511Crofton}
    \int_{A(n,q)} V_{q-j}(K \cap E) \,\mu_q (dE) 
    =
    c_{q-j,n}^{n-j,q}\, V_{n-j}(K), \quad K \in \mathcal{K}^n,
\end{equation}
valid for $q \in \{0,...,n\}$ and $j \in \{0,...,q\}$ with 
\begin{equation}
    c_{s_1,s_2}^{r_1,r_2}:= \frac{r_1! \kappa_{r_1}}{s_1! \kappa_{s_1}}\frac{r_2! \kappa_{r_2}}{s_2! \kappa_{s_2}}.    \label{eqGeoCroftonKonstant}
\end{equation}
Numerous variants and generalizations of \eqref{thmGeo511Crofton} can be found in the literature, see the overview \cite{HugSchneider}--the following list exemplifies the variety without claiming completeness:  
the underlying Euclidean space can be replaced by a Finsler space \cite{Bernig07,CroftonFinslerSpaces}, 
the results for intrinsic volumes can be extended to support measures \cite{Glasauer97}, the 
set class $\mathcal K^n$ can be generalized, for example to $\mathcal U_{PR}$-sets \cite{Rataj02}, a family of certain finite unions of sets with positive reach,   and the invariant integration in \eqref{thmGeo511Crofton} can be replaced by an invariant integration over translations of a fixed flat, leading to \emph{translative Crofton formulae} \cite{GoodeyWeil,Rataj02}. 
Our new result \eqref{eqNy3} can be considered as a generalization in the spirit of the last example, as we also replace integration over the full space $A(n,q)$ by integration over a geometrically meaningful subset.

\chapter{Rotational Crofton Formulae with a Fixed Subspace}
The main result of this paper is the following theorem. It uses the constant 
\begin{equation}
    \alpha_{n,k,q,r}
    :=
    \frac{\omega_{k-q-r} \cdots \omega_{k-q}}{\omega_{n-q-r} \cdots \omega_{n-q}} \prod_{j=0}^{r-1} \frac{\omega_{n-j}}{ \omega_{k-j}}     \label{eqGeoAlphaKonstant}.
\end{equation}
For $r=0$ the last product in \eqref{eqGeoAlphaKonstant} is set equal to $1$ by convention.

\begin{thm}[Rotational Crofton formulae with a fixed subspace]\label{thmGeoRotationalCroften}
Let  $n,r,k\in \N_0$ with $r+1\le k\le n$ be given and fix a subspace $L_0\in G(n,r)$. Then, for $j=0,\ldots,k-(r+1)$, 
\begin{equation}\label{eq:CroftAxis}
	\int_{G(L_0,k)} \varphi_{L,q}^{L_0}(K\cap L)\,\nu_k^{L_0}(dL)=V_{n-j}(K),\qquad K\in \mathcal{K}^n,
\end{equation}
holds with 
\begin{align}
\varphi_{L,q}^{L_0}(K\cap L)
	=\alpha_{n,k,q,r}^{-1}c^{q-j,n}_{n-j,q}\int_{A(L,q)} V_{q-j}(K\cap E)\,D(E,L_0)^{n-k}\,\mu_q^L(dE).
	\label{eq:measureFunct_RotCrofton}
\end{align}
Here, $D(E,L_0)$ is given by \eqref{eq:DefD}, the leading constant is defined in 
\eqref{eqGeoCroftonKonstant} and  \eqref{eqGeoAlphaKonstant}, and $q$ can be chosen in $\{j,\ldots,k-(r+1)\}$. 
\end{thm}
The special case $r=0$ of Theorem \ref{thmGeoRotationalCroften} poses no constraints on the linear spaces in $G(L_0,k)=G(\{o\},k)=G(n,k)$ and states that 
\begin{equation}\label{eq:rotCroftr=0}
	\int_{G(n,k)} \varphi_{L,q}(K\cap L)\,\nu_k^{L_0}(dL)=V_{n-j}(K),\qquad K\in \mathcal{K}^n,
\end{equation}
where 
\begin{align}
	\varphi_{L,q}(K\cap L)
	=%\alpha_{n,k,q,0}^{-1}c^{q-j,n}_{n-j,q}
	\frac{\omega_{n-q}}{\omega_{k-q}}c^{q-j,n}_{n-j,q}
	\int_{A(L,q)} V_{q-j}(K\cap E)d(o,E)^{n-k}\,\mu_q^L(dE)
	\label{eq:measureFunct_RotCroftonr=0}
\end{align}
for $q\in\{j,\ldots,k-1\}$. 
This is the precise statement of the rotational Crofton formula in \cite{EvaAuneau2010} already mentioned qualitatively in the introduction in \eqref{eqNy2}. There thus appear to be $k-j$ different measurement functions. However, $\varphi_{L,j}=\cdots=\varphi_{L,k-1}$, so all these measurement functions coincide. This fact, although known, appears to be unpublished (see, however, the forthcoming publication \cite{EvaMarkusBog}). It also follows from Theorem \ref{Thm:Uniqueness}  in Section \ref{sec:measureFct}, which establishes the independence of the measurement functions of $q$ for general  $r\ge 0$.

The proof of Theorem \ref{thmGeoRotationalCroften} will be given below. It is obtained by combining the classical Crofton formula 
\eqref{thmGeo511Crofton} with the following Blaschke--Petkantschin result.

\begin{thm}\label{thm:BP}
		Let $n,q,r,k\in \N_0$ with $q+r+1\le k\le n$ be given and fix a subspace $L_0\in G(n,r)$. Then, for
		any  measurable function $f:A(n,q)\to [0,\infty)$ we have 
		\begin{align}\nonumber
			&\int_{G(L_0,k)}\int_{A(L,q)} f(E)  D(E,L_0)^{n-k}\, \mu_q^L(dE)\,\nu_k^{L_0}(dL)
 \\ \quad &=
			\alpha_{n,k,q,r}\int_{A(n,q)}f(E) \,\mu_q(dE),\label{eq:nyBP}
		\end{align} 
		with the constant  \eqref{eqGeoAlphaKonstant} and $D(E,L_0)$ given by \eqref{eq:DefD}. 
	\end{thm}
\begin{proof}
	We start with the special case $q=0$. Equation  \eqref{thmGeo722} implies 
\begin{align}\label{eqGeoq=0}
 &\int_{A(n,0)} f(E) \,\mu_q(dE)
 = \int_{\mathbb{R}^n} f(\{x\}) \,\lambda(dx) \nonumber
    \\&=
    \frac{b_{n-r,1}}{b_{k-r,1}}
     \int_{G(L_0,k)} \int_{L}f(\{x\}) \nabla_{1,r}(x,L_0)^{n-k}\,\lambda_L(dx) \, \nu_{k}^{L_0}( dL).
\end{align}
Since  $\nabla_{1,r}(x, L_0)=D(\{x\},L_0)$, $\mu_0^L$ is the image measure of 
 $\lambda_L$ under the identification $x \mapsto \{x\}$, 
 and $ \frac{b_{n-r,1}}{b_{k-r,1}} = \alpha_{n,k,0,r}^{-1}$ the relation \eqref{eqGeoq=0} is equivalent to the claim 
 when $q=0$. 
 
Now assume $q>0$ and consider
\begin{align*}
    g: (\mathbb{R}^n)^{q+1} &\rightarrow \mathbb{R},
    \\
    \x &\mapsto  f( \aff(\x)) \Delta_q(\x)^{q-n} h(\x),
\end{align*}
with $h(\x)=\prod_{i=1}^{q+1} 1_{B^n}\big(x_i-p(o|\aff(\x))\big)$, defined for all $\x \in (\mathbb{R}^n)^{q+1}$ such that $\aff(x) \in A(n,k)$. Using \eqref{thmGeo727} we get
\begin{align}\nonumber
    &\int_{(\mathbb{R}^n)^{q+1}} g(\x) \,\lambda^{q+1}(d\x)
  \\  \nonumber &\,\quad=
   b_{nq} (q!)^{n-q} \int_{A(n,q)} \int_{E^{q+1}}     f(\aff(\x)) \Delta_q(\x)^{q-n} h(\x)
    \Delta_q(\x)^{n-q} \, \lambda_E^{q+1}(d\x) \, \mu_q(dE)
    \\ \nonumber
    &\ \quad=
    b_{nq} (q!)^{n-q} \int_{A(n,q)}f(E) 
     \int_{E^{q+1}}     h(\x)
  \, \lambda_E^{q+1}(d\x) \, \mu_q(dE)
  \\
 &\ \quad=
  b_{nq} (q!)^{n-q} \kappa_{q}^{q+1} \int_{A(n,q)}f(E) \,\mu_q(dE).\label{eqPart1}
\end{align}
On the other hand \eqref{thmGeo722} with $q$ replaced by $q+1$ gives
\begin{align}\label{eqPart2}
     \int_{(\mathbb{R}^n)^{q+1}} g (\x) \,\lambda^{q+1}(d\x)%&=\int_{(\mathbb{R}^n)^{q+1}} f(\aff(\x)) \Delta_q(\x)^{q-n} h(\x) \, \lambda^{q+1}(d\x)\\
     &=
 \frac{b_{n-r,q+1}}{b_{k-r,q+1}}     \int_{G(L_0,k)}I(L) \,\nu_{k}^{L_0}(dL),
\end{align}
with 
\begin{equation*}
   I(L):=  \int_{L^{q+1}} g(\x) \nabla_{q+1,r}(\x, L_0)^{n-k} \,\lambda_L^{q+1}(d\x).
\end{equation*}
Using  $L_0 \subseteq L$, we can identify $L$ with $\mathbb{R}^k$. An application of \eqref{thmGeo727} in $\R^k$ then yields
\begin{align*}
    I(L)&=b_{k,q}(q!)^{k-q} \int_{A(k,q)} \int_{E^{q+1}} g(\x) \Delta_q(\x)^{k-q} \nabla_{q+1,r}(\x, L_0)^{n-k} \, \lambda_E^{q+1}(d\x) \, \mu_q (dE). 
    \end{align*}
    Inserting the definition of $g$ gives    
    \begin{align*}
    I(L)&=
   b_{k,q}(q!)^{n-q}  \int_{A(k,q)} \int_{E^{q+1}} f\big(\aff(\x)\big) \left(\frac1{q!}\frac{\nabla_{q+1,r}(\x, L_0)}{\Delta_q(\x)}\right)^{n-k}\\
   &\qquad \times h(\x)  \, \lambda_E^{q+1}(d\x) \, \mu_q (dE)
  \\
   &=b_{k,q}(q!)^{n-q} \kappa_q^{q+1}  \int_{A(k,q)} f(E) D(E,L_0)^{n-k}
  \, \mu_q (dE),
\end{align*}
where Lemma \ref{lemD(E,L)} was used in the last step. 

Comparing \eqref{eqPart1} with \eqref{eqPart2}, using the explicit form of $I(L)$ just derived, 
shows the claim \eqref{eq:nyBP}. To simplify the constant, we used 
\begin{align*}
   \frac{b_{n,q}}{b_{k,q}} \frac{b_{k-r,q+1}}{b_{n-r,q+1}} 
   =
   \frac{\omega_{n-q+1} \cdots \omega_n}{\omega_{n-r-q}\cdots \omega_{n-r}}
   \cdot \frac{\omega_{k-r-q} \cdots \omega_{k-r}}{\omega_{k-q+1} \cdots \omega_k}
   =
    \frac{\omega_{k-q-r} \cdots \omega_{k-q}}{\omega_{n-q-r} \cdots \omega_{n-q}} \prod_{j=0}^{r-1} \frac{\omega_{n-j}}{ \omega_{k-j}},
\end{align*}
which follows by induction in $r=0,1, \ldots k-q-1$, if we define $\prod_{j=0}^{r-1} \frac{\omega_{n-j}}{ \omega_{k-j}}:=1$ when $r=0$. 
This concludes the proof. 
\phantom\qedhere
\end{proof}

We can now give a proof of the main result. 
\begin{proof}[Proof of Theorem \ref{thmGeoRotationalCroften}]
With the assumptions of Theorem \ref{thmGeoRotationalCroften}, put $f(E)=V_{q-j}(K\cap E)$ in the \\Blaschke--Petkantschin formula (Theorem \ref{thm:BP}) and observe that the right-hand side is a classical Crofton integral that can be evaluated using \eqref{thmGeo511Crofton}. 
\phantom\qedhere
\end{proof}

\chapter{The measurement functions}\label{sec:measureFct}
We collect properties and simplified representations of the measurement function defined in Theorem \ref{thmGeoRotationalCroften}. To avoid unnecessary repetitions, we put
\begin{equation}
c_0(q):=  
 \alpha_{n,k,q,r}^{-1}c^{q-j,n}_{n-j,q}
 \label{eq:defC0}
\end{equation}
for the leading constant in \eqref{eq:measureFunct_RotCrofton}, thinking of all dimensions other than $q$ as fixed. 

The main result in this section is the independence of $\varphi_{L,q}^{L_0}$ from $q$ given in  Theorem \ref{Thm:Uniqueness}. 
With this independence in mind,  we can simply choose one value of $q$ and provide simplified expressions. This will be done in statements (ii) and (iii) of Proposition \ref{thmGeoReduktionMeasurementFunction} below. We start by stating the uniqueness result.

\begin{thm}\label{Thm:Uniqueness}
		Let the setting be as in Theorem \ref{thmGeoRotationalCroften}. Then, 
		$\varphi_{L,q}^{L_0}(L \cap K)$ is independent of 
		$q\in  \{j,\ldots, k-(r+1)\}$.
	\end{thm}
\begin{proof}

 Let the assumptions of Theorem \ref{thmGeoRotationalCroften} be satisfied. If $\{j, \ldots ,k-(r+1)\}$ is a singleton, the claim is trivial. Thus, we may assume  $j <k-(r+1)$ and 
		it is enough to show 
			\begin{equation}
				\varphi_{L,q+1}^{L_0}(\cdot)=\varphi_{L,q}^{L_0}( \cdot)
				\label{eq:whatwewant}
			\end{equation}	
				 for every fixed $q\ge j$ such that $q +1 \leq k-(r+1)$. 
		
		Fix $E_1 \in A(L,q+1)$. Applying \eqref{thmGeo511Crofton} to the convex body $K \cap E_1$ in $E_1$ gives
		\begin{equation*}
			V_{(q+1)-j}(K \cap E_1) = c^{q-j,q+1}_{q-j+1,q} \int_{A(E_1,q)} V_{q-j}(K \cap E) \,\mu_q^{E_1}(dE).
		\end{equation*}
Combining this with 
\eqref{eq:measureFunct_RotCrofton}
yields 
\begin{align*}
    \varphi_{L,q+1}^{L_0}(K \cap L)
			&=c_1
\int_{A(L,q+1)} 
   \int_{A(E_1,q)} V_{q-j}(K \cap E) \,\mu_q^{E_1}(dE)
   \\& \qquad\times D(E_1,L_0)^{n-k}\,\mu_{q+1}^L(dE_1),
\end{align*}
where $c_1= c_0(q) {\frac{\omega_{n-r-(q+1)}}{\omega_{k-r-(q+1)}}} \frac{\omega_{k-q}}{\omega_{n-q}}$. 
Note that all flats involved are subsets of the $k$-dimensional subspace $L$, so we may intechange integrals due to \cite[Thm.~7.1.2]{GulBog},  applied in $L$. Identifying $L$ with $\R^k$, we get
	\begin{equation}
			\varphi_{L,q+1}^{L_0}(K \cap L)
			=
			c_1 \int_{A(L,q)} V_{q-j}(K \cap E) h(E) \,\mu_q^L(dE),
			\label{eq:q+1q}
		\end{equation}
		where
		\begin{equation*}
			h(E) =\int_{A(E,q+1)} D(E_1,L_0)^{n-k} \,\mu_q^{E}(dE_1)
		\end{equation*}
   is to be understood in $\R^k$. All orthogonal complements and unit spheres that will appear in the calculation of 
   $h(E)$ will therefore also be understood relative to $\R^k$.
	Writing $E= M+z$ with $M \in G(k,q)$ and $z \in M^\perp$, \cite[eq.~(13.14)]{GulBog} gives
		\begin{align*}
			h(M+z)
			&=\int_{G(M,q+1)}
			D(L_1+z,L_0)^{n-k} \,\nu_{q+1}^M(dL_1)
			\\&=
			\omega_{k-q}^{-1}
			\int_{S^{k-1} \cap M^\perp }
			D(M+\myspan(u) +z,L_0)^{n-k} \,{\mathcal{H}^{k-q-1}(du)} 
			\\&=\omega_{k-q}^{-1}
			D(E,L)^{n-k}
			\int_{S^{k-1} \cap M^\perp }
			\| p(u |(L_0+M+z)^\perp )  \|^{n-k} \,\mathcal{H}^{k-q-1}(du),
		\end{align*}
	  where we used  \eqref{eqOneDimMore} at the last equality sign.

	  We will now show that $z \neq o$ implies
	  		\begin{align}\nonumber
	  	I_z(M)&:=  \int_{S^{k-1} \cap M^\perp }
	  	\| p(u |(L_0+M+z)^\perp )  \|^{n-k} \, {\mathcal{H}^{k-q-1}}(du)
	  	\\&=	 \frac{\omega_p}{\omega_{n-k+p}}\omega_{n-q},
	  	\label{eq:IM}
	  \end{align}
  with $p = \dim ( L_0+M+z)^\perp$.

  To show relation \eqref{eq:IM} we apply \eqref{eqGEOSphericalFormula} in the $(k-q)$-dimensional space $M^\perp$
  	with $v=z/\|z\|$ using that $p(v|(M+L+z)^\perp)=o$ to obtain
\begin{align}
	I_z(M)= c'\,\,\int_{S^{k-1} \cap (M+ z)^\perp }
	\| p(w |(L_0+M+z)^\perp )  \|^{n-k} \, \mathcal{H}^{k-q-2}(dw), 
	\label{eq:Ifast}
\end{align}
where 
\[
c':=\int_{-1}^1 (1-t^2)^{\frac{n-q-3}{2}} \,dt=\frac{\omega_{n-q}}{\omega_{n-q-1}}
\]
has been  evaluated by setting $f(u)=1$ in  \eqref{eqGEOSphericalFormula}.  
  
  For fixed $d,j,p\in \N$ with $p\le d$, and $L'\in G(d,p)$ the result in \cite[Lemma 1]{EvaAuneau2010} gives 
\begin{align}\label{eq:brugesIAppendix}
\int_{S^{d-1}} \|p(u|L')\|^j\, du
&=\frac{\omega_{d-p}\omega_p}{2}
\int_{0}^1 t^{\frac{j+p}2-1}(1-t)^{\frac{d-p}{2}-1}\, dt 
=\frac{\omega_{d+j}}{\omega_{p+j}}\omega_p.
\end{align}
Using this with  $d=k-(q+1)$, $j=n-k$ and $p=\dim(L_0+M+z)^\perp$ in \eqref{eq:Ifast} yields the claim \eqref{eq:IM}.

  This claim \eqref{eq:whatwewant} follows  by inserting \eqref{eq:IM} into $h(\cdot)$ and the result into \eqref{eq:q+1q} and observing that $p=k-r-(q+1)$ holds for almost every $M$. Indeed,
\begin{align*}
p=\dim(L_0+M+z)^\perp 
%&=k-\dim(L_0+M+z)\\&
&=k-\dim(L_0)-\dim(M+z)+\dim(L_0\cap (M+z))
\\&=k-r-(q+1)+0
\end{align*} 	
a.e., as an invariance argument yields 		
\begin{align*}
	& \int_{G(k,q)}\int_{M^\perp\cap RB^k} 1_{\dim(\myspan(M+z)\cap L_0)>0} \,\lambda_{M^\perp}(dz) \,\nu_q(dM)
	\\&=R^{k-q}\kappa_{k-q}
	\int_{G(k,q+1)}1_{\dim(M'\cap L_0)>0} \, \nu_{q+1}(dM')=0
\end{align*}
for all $R>0$, where we used \cite[Lemma 4.4.1]{SchneiderRoed} and $r+(q+1)<k$. This concludes the proof and the theorem is shown.
\phantom\qedhere
	\end{proof}

The following Proposition states properties of the measurement functions. Item (i) describes general characteristics.
 With Theorem \ref{Thm:Uniqueness} in mind, we can focus on one value of $q\in \{j,\ldots k-(r+1) \}$ when providing reductions of $\varphi_{L,q}^{L_0}$. 
Item (ii) in Proposition \ref{thmGeoReduktionMeasurementFunction} below, 
chooses therefore the minimal possible $q$: it shows that  the measurement function for $q=j$ can be interpreted as average of weighted projections onto $j$-dimensional linear subspaces that in turn are weighted according to their position relative to $L_0$. 
 Item (iii) expresses simply the case $j=0$, written for convenience.
That the right-hand side of \eqref{eq:measureVn} yields a rotational Crofton formula with axis for the volume can also be seen directly from \eqref{thmGeo722}. 

\begin{prop}\label{thmGeoReduktionMeasurementFunction}
	Let the assumptions of Theorem \ref{thmGeoRotationalCroften} be satisfied and fix $q\in\{j,\ldots,k-(r+1)\}$.
	When $L \in G(L_0,k)$  the following statements hold. 
	\begin{enumerate}
		\item[(i)] $\varphi_{L,q}^{L_0}$ is an additive 
		functional on the convex bodies in $L$. It is positive homogeneous of degree $n-j$ and invariant under all rotations fixing $L$ and $L_0$. 
		\item[(ii)] We have 
		\begin{align}
			\nonumber
		\varphi_{L,j}^{L_0}(K')&=c_0(j)
		\int_{G(L,j)} [M,L_0]^{n-k}\int_{K'|M^\perp}
		\\&\qquad \qquad \times 
	     d(z,M+L_0)^{n-k} \,\lambda_{M^\perp}(dz)\, \nu_j^L(dM),
	     \label{eq:qsmallest}
		\end{align}
		for all convex bodies $K'\subset L$. 
		\item[(iii)] If $j=0$ (rotational integral for the volume) this simplifies to 
		\begin{align}
	\varphi_{L,q}^{L_0}(K')&=\frac{\omega_{n-r}}{\omega_{k-r}} \int_{K'} d(x,L_0)^{n-k} \, \lambda_{L}(dx),
	\label{eq:measureVn}
\end{align}	      
		for all convex bodies $K'\subset L$. 
	\end{enumerate}
    \renewcommand{\theenumi}{\arabic{enumi}}%
\end{prop}
\begin{proof}
	The proof of (i) is straightforward. It uses additivity, homogeneity and motion invariance of the   intrinsic volumes, see, e.g.,~\cite[Section 3.3]{LectureNotesConvexGeometry} and the fact that $\varphi_{L,q}^{L_0}$ is defined via an invariant integral. 
	
	To show (ii) we may assume $q=j$ due to Theorem 
    \ref{Thm:Uniqueness}. Then \eqref{eq:measureFunct_RotCrofton}, and an application of  \eqref{eqGEOAffineIntegrationToLinearplusShift} 
	and \eqref{eq:DefDnew} yield
		\begin{align*}
	\varphi_{L,j}^{L_0}(K')&=c_0(j)
	\int_{G(L,j)} [M,L_0]^{n-k}\int_{M^\perp\cap L}
	\\&\qquad \times 
	 V_0(K'\cap (M+z))d(z,M+L_0)^{n-k} \, \lambda_{M^\perp}(dz)\, \nu_j^L(dM).
\end{align*}
The assertion now follows, observing that $V_0(K'\cap (M+z))=\1_{K'|M^\perp}(z)$.

	Finally, (iii) is a direct consequence of \eqref{eq:qsmallest} with $j=0$, observing that  $G(L,0)$ is a singleton containing only $\{o \}$.
 \phantom\qedhere
\end{proof}   

We remark that \eqref{eq:qsmallest} can also be expressed in terms of the radial function  of $K'|M^\perp$. 
The \emph{radial function}  of a convex body $K\subset\R^n$ with $o\in K$ is given by $\rho_{K}(u):= \sup \{ t \geq 0 \,: \,t\cdot u \in K \}$ for $u \in S^{n-1}$. 
Introducing spherical coordinates one sees that 
\[
\int_{K} f(z)\, dz=\int_{S^{n-1}} \int_0^{\rho_K(u)}f(ru) r^{n+1}\, dr\, \cH^{n-1}(du)
\]
holds for every non-negative measurable function    $f$. 
Applying this in $M^\perp$ with $f(z)=d(z,M+L_0)^{n-k}$ and $K=K'|M^\perp$ yields
		\begin{align*}
	\varphi_{L,j}^{L_0}(K')&=\frac{c_0(j)}{n-j}
	\int_{G(L,j)} [M,L_0]^{n-k}\int_{S^{n-1}\cap M^\perp\cap L}
	\\&\qquad \qquad \times 
	d(u,M+L_0)^{n-k} \rho_{K'|M^\perp}^{n-j}(u)\,\cH^{k-j-1}(du)\, \nu_j^L(dM)
	\label{eq:qsmallest}
\end{align*}  
provided that $o\in K'$. This gives an alternative representation of \eqref{eq:measureVn} in the case $j=0$, namely,
\begin{align*}
    \varphi_{L,0}^{L_0}(K')&=\frac{\omega_{n-r}}{n\omega_{k-r}}
	\int_{S^{n-1}\cap L}
		d(u,L_0)^{n-k} \rho_{K'}^{n}(u)\,\cH^{k-1}(du)
\end{align*}
which may be advantageous when the radial function is easily accessible. 
\medskip

We will now briefly discuss  which indices in Theorem \ref{thmGeoRotationalCroften} lead to new formulae and which are already established in the literature. 
As already outlined directly after the statement of Theorem \ref{thmGeoRotationalCroften}, our main result reduces  when $r=0$ to the rotational Crofton formulae in \cite{EvaAuneau2010} and \cite{GUALARNAU2010}. 
If we further restrict considerations to $n=3$ the resulting formulas date even further back as then the results are closely related to classical estimators: 
%(for $k=1$ and hence $j=0$, i.e.~volume estimation) 
the \emph{nucleator} \cite{Nucleator}  and the \emph{rotator} \cite{EstRotator}, both already mentioned in the introduction. 
Of surface estimators in the three-dimensional space, we have the indices 
$n=3,k=2,r=0,q=1,j=1$ where our formula is known as \emph{integrated surfactor} or \emph{flower estimator}.

When $r \neq 0$ in the 
three-dimensional setting, excluding 
the trivial case $k=3$ only leaves us with the indices $r=1$, $k=2$, $q=0$ and $j=0$ yielding an integral relation for the volume from planar sections. 
The resulting estimator is well-known and goes under the name \emph{vertical estimator} and to the authors' knowledge dates back to 1993, see \cite{EstRotator}. 

In fact, all admissible estimators in the three-dimensional setting are already established in the rich literature in stereology and local stereology.  
When $n=4$ the first new results appear. 
Due to \eqref{eq:measureVn} we get three different formulas for the volume of $K \in \mathbb{R}^4$ and one formula relating to 
the surface area of $K \in \mathbb{R}^n$. An example could be the choice of indices $n=4,k=3,r=q=j=1$ resulting in
    \begin{align*}
        V_3(K) =c_0(1)
        \int_{G(L_0,3)} \int_{A(L,1)} V_0(K \cap E) d(o,E)\sin \sphericalangle (E,L_0) \,\mu_2^L(dE) \,\nu_2^{L_0}(dL),
    \end{align*}
   with $\sphericalangle(E,L_0)$ being the minimal angle between the two lines $E$ and $L_0$. This relates the surface area of a 
   4-dimensional convex body $K$ to a double integration, where  the inner integral is over lines passing through $K\cap L$ weighted by the relative angle to $L_0$ and their distance from the origin.

\chapter{Rotational Crofton formulae with axis for other intrinsic volumes}\label{sec:limitations}

We end this paper with a short discussion of the possibility to weaken the constraints on the indices in our main result.
For a given $r\le n-1$  Theorem \ref{thmGeoRotationalCroften} 
states rotational Crofton formulae with axis for $V_m$, $m\in
\{r+1,\ldots,n\}$. Strictly speaking, 
the result for the homogeneity degree  $m=r+1$ is trivial, as it can only be achieved by choosing $k=n$, 
meaning that the left side of 
\eqref{eq:CroftAxis} depends on all of $K$ instead of a lower dimensional section of $K$. 
Hence, only $V_m$, $m\in
\{r+2,\ldots,n\}$, can be obtained 
from lower dimensional sections in 
Theorem \ref{thmGeoRotationalCroften}.  
The question arises if this limitation is due to our method of proof or if it is a geometrical limitation that cannot be overcome 
with other methods either. The following proposition shows that the latter is the case. In this proposition, we exclude the case $k=n$, which just has been discussed. 

\begin{prop}\label{prop:Indices}
Let  $n,r,k\in \N_0$ with $r+1\le k\le n-1$ be given. Fix a subspace $L_0\in G(n,r)$ and $m\in
\{0,\ldots,r+1\}$. Then, there is no function $\varphi$ satisfying 
\begin{equation}\label{eq:CroftNot}
	\int_{G(L_0,k)} \varphi(K\cap L)\,\nu_k^{L_0}(dL)=V_{m}(K)
\end{equation}
for all $K\in \mathcal{K}^n$. 
\end{prop}
\begin{proof}
To show that \eqref{eq:CroftNot}  is impossible, it is enough to find two convex bodies $K_1,K_2$ such that $K_1 \cap L = K_2 \cap L$ for almost all $L \in G(L_0,k)$ and $V_{m}(K_1)\neq V_{m}(K_2)$. 

If $m=0$ such an  example is given by the singleton $K_1=\{x\}$ with $x\not \in L_0$ and the empty set $K_2=\emptyset$. For $m>0$ choose 
a linear subspace $L'\subset L_0$ of dimension $m-1$ and a vector 
$u\in L_0^\perp\cap S^{n-1}$. Put
\begin{align*}
    K_1 &= (B^n\cap L')+ \conv\{o,u\},
    \\
    K_2&= B^n\cap L'. 
\end{align*}
Then 
$K_1 \cap L = K_2 \cap L$ for almost all $L \in G(L_0,k)$ as $k\le n-1$. 
As $\dim K_1=m$ and $\dim K_2=m-1$ we have 
$V_m(K_1)>0$ and $V_m(K_2)=0$. This yields the claim.  
\phantom\qedhere
\end{proof}

Another possibility for rotational Crofton formulae of the type discussed in this paper is when integrating over $G(L_0,k)$ for a fixed axis $L_0 \in G(n,r)$ where $k \le  r$. This setting is not of practical interest, as any part of $K$ outside $L_0$ cannot be observed by the sections considered. Hence, we only can hope to estimate intrinsic volumes of bodies $K \in \{K'\in \mathcal{K}^n: K'\subset L_0\}$. For such  $K$ one can of course identify $G(L_0,k)$ with $G(r,k)$ and just utilize the classical Crofton formulae with suitable parameters.

%\bibliographystyle{abbrv}
%\bibliography{bib}
%\newpage

\appendix
\makeatletter
\renewcommand{\@chapapp}{Appendix}
\makeatother
\renewcommand{\theequation}{\thechapter.\arabic{equation}}
\setcounter{equation}{0}

\chapter*{Appendix A. Vertical sections for intrinsic volumes}
In this appendix we will give a proof of  the vertical sections relation \eqref{eqNy4} based on a suitable Blaschke--Petkantschin formula. For a flat $E \in A(n,q)$ we let $\lin (E):=E-p(o|E)$ denote the unique linear subspace of the same dimension which is parallel to $E$. It should be noted that $\lin(E)$ and  $\myspan (E)$ typically differ. As  $\myspan (E) = \lin(E)+\myspan (\{x\} )$ holds for arbitrary $x \in E$, 
they coincide if and only if $E$ is a linear subspace.  We will now prove the following result on vertical sections. 

\begin{thm}[Vertical sections]\label{thmGeoAffineRotationalCroften}
Let  $n,r,k\in \N_0$ with $r+1\le k\le n$ be given and fix a subspace $L_0\in G(n,r)$. Then, for $j=0,\ldots,k-r$, 
\begin{equation*}
    \int_{G(L_0,k)}\int_{L^\perp } \tilde{\varphi}_{L+x,q}^{L_0}(K \cap (L+x) ) \,\lambda_{L^\perp} (dx) \,\nu_k^{L_0}(dL)
    =
    V_{n-j}(K),\qquad K\in \mathcal{K}^n,
\end{equation*}
holds with
\begin{equation*}
 \tilde{\varphi}_{L+x,q}^{L_0}(K \cap (L+x) ) = d_{n,k,r,j}(q) \int_{A(L+x,q)} V_{q-j}\big(K\cap E \big) [\lin(E),L_0]^{n-k} \,\mu_q^{L+x}(dE).
\end{equation*}
Here $q$ can be chosen in $\{j,\ldots,k-r\}$. The leading constant is
\[
d_{n,k,r,j}(q)=
\frac{b_{n-r,q}}{b_{k-r,q}}
\frac{b_{k,q}}{b_{n,q}} c^{q-j,n}_{n-j,q}= 
c_0(q) \frac{\omega_{k-r-q}}{\omega_{n-r-q}}
\]
with $c_0(q)$ given by \eqref{eq:defC0}. 
\end{thm}
For $r=0$, the classical Crofton formula  \eqref{thmGeo511Crofton}, applied in the $k$-dimensional flat $L+x$, allows one to simplify the measurement function:  
\[
\tilde{\varphi}_{L+x,q}^{L_0}(K \cap (L+x) ) =c^{k-j,n}_{n-j,k}
V_{k-j}(K\cap (L+x)). 
\]
Hence, the special case $r=0$ of Theorem \ref{thmGeoAffineRotationalCroften} reduces  to the classical Crofton formula.

The vertical section formula in Theorem \ref{thmGeoAffineRotationalCroften} from $k$-dimensional sections allows for the estimation of $V_m$ with $m\in \{n-k+r,\ldots,n\}$. For fixed $r$, the integral relation with vertical sections thus gives access to one more intrinsic volume of $K$, namely $V_{n-k+r}(K)$, as compared to the 
rotational Crofton formula with a fixed subspace (Theorem \ref{thmGeoRotationalCroften}). 
This is one of the reasons that the concept of vertical sections is so well-established in the applied literature. 

We will now give a proof of Theorem \ref{thmGeoAffineRotationalCroften}  using  a suitable Blaschke--Petkantschin formula, which we  will also prove. We mentioned already in the introduction that \cite{Baddeley84} showed vertical section formulae for Hausdorff measures of sets in $\R^n$.
The rather contracted proof in that paper was based on the coarea formula. 
Interestingly enough, the result to come has the same structure as in \cite[eq.~(6)]{Baddeley84} if one replaces a Hausdorff measure for an appropriate intrinsic volume. 

The proof of Theorem \ref{thmGeoAffineRotationalCroften} parallels the one in Section 3.  To state the relevant 
Blaschke--Petkantschin formula, we will use  the following relation: 
\begin{equation}\label{eqAffineBPhelp}
    \int_{L^\perp}  \int_{A(L+x,q)} f(E) \,\mu_q^{L+x}( dE) \, \lambda_{L^\perp}(dx)
    =
    \int_{G(L,q)} \int_{ M^\perp} f(M+x)  \,\lambda_{M^\perp}(dx) \,\nu_q^L(dM), 
\end{equation}
valid for $L \in G(n,q)$ and any measurable function $f:A(n,q)\rightarrow [0, \infty)$, see \cite[eq.~(13.13)]{GulBog}.

We now formulate an affine Blaschke--Petkantschin formula with a fixed subspace, the  proof of which being inspired by 
\cite[Lemma 2.2]{StoGen13}.  By ``affine'' we here mean that we want to  choose an invariant flat parallel to the fixed subspace $L_0$. This can be done in a two-step procedure, where one first chooses a subspace $L$ containing $L_0$ and afterward moves the subspace invariantly by a vector $x \in L^\perp$. 
\begin{thm}
\label{thm:verticalBP} 
Let $n,q,r,k \in \mathbb{N}_0$ with $q+r \leq k \leq n$ be given and fix a subspace $L_0 \in G(n,r)$. Then, for any measurable function $f \,: \,A(n,q)\rightarrow [0, \infty)$ we have
\begin{align*}
    &\int_{G(L_0,k)}\int_{L^\perp } \int_{A(L+x,q)} f(E) [\lin (E),L_0]^{n-k} \,\mu_q^{L+x}(dE) \,\lambda_{L^\perp} (dx) \,\nu_k^{L_0}(dL)
    \\
    &=      \frac{b_{k-r,q}}{b_{n-r,q}} \frac{b_{n,q}}{b_{k,q}} \int_{A(n,q)} f(E) \,\mu_q(dE).
\end{align*}
\end{thm}
\begin{proof}
For any measurable and positive function $g \,: \,G(n,q)\rightarrow [0, \infty)$ we have 
\begin{equation}\label{eqAffineWhyEasy}
\kappa_n^q \int_{G(n,q)} g(L) \,\nu_q(dL)
=
   \int_{(B^n)^q} g( \myspan \x )   \,\lambda^q_n (d \x ), 
\end{equation}
as the image measure of the restriction of $\lambda^q_n $ to $(B^n)^q$ under the mapping $\x\mapsto\myspan \x$ is a rotationally invariant measure on $G(n,q)$ and thus a multiple of $\nu_q$. Relation  \eqref{thmGeo722} with $f(\x)=g(\aff \x)\1_{(B^n)^q}(\x)$ gives  
\begin{align}
\kappa_n^q \int_{G(n,q)} g(L) \,\nu_q(dL)
=
   \frac{b_{n-r,q}}{b_{k-r,q}}\int_{G(L_0,k)}  I(L)  \,\nu_k^{L_0}(dL),
   \label{eq:zrech}
\end{align}
where we used the abbreviation 
\[
I(L)=
\int_{(B^n \cap L)^q} g( \myspan \x)\nabla_{q,r}(\x ,L_0)^{n-k}\lambda_L^q (d \x ). 
\]
Now \cite[Thm.~7.2.1]{GulBog}, applied in $L$,
allows us to write 
\begin{align*}
    I(L)
   & =
  b_{k,q} \int_{G(L,q)} \int_{(B^n\cap M)^q} g (\myspan \x)   \nabla_{q,r}(\x ,L_0)^{n-k} \nabla_{q} (\x)^{k-q} \,\lambda_{M}^q (d \x )
  \,   \nu_q^{L}(d M)
     \\
     &=
    b_{k,q} \int_{G(L,q)} g(M) \int_{(B^n\cap M)^q}
   \nabla_{q,r}(\x ,L_0)^{n-k} \nabla_{q} (\x)^{k-q} \,\lambda_{M}^q (d \x )
    \, \nu_q^{L}(d M),
\end{align*}
as $\myspan \x=M$ holds almost surely. The latter also implies 
\begin{equation*}
    \nabla_{q,r}(\x ,L_0) = [\myspan \x ,L_0] \nabla_q(\x) 
    =
    [M ,L_0] \nabla_q(\x), 
\end{equation*}
almost surely. Hence, the inner integral coincides with 
\begin{align*}
     &[M,L_0]^{n-k} 
     \int_{(B^n\cap M)^q}
       \nabla_{q} (\x)^{n-q} \,\lambda_{M}^q (d \x )
     \\&=
      [M,L_0]^{n-k}  \int_{G(n,q)}\int_{(B^n\cap M)^q}  \nabla_{q}(\x )^{n-q} \, \lambda_q^q (d \x ) \, \nu_k(dM)
      \\&=
      [M,L_0]^{n-k}  \frac{1}{b_{n,q}}  \int_{(B^n)^q} \lambda^q (d\x)
      \\
      &=
    [M,L_0]^{n-k}  \frac{\kappa_n^q}{b_{n,q}},
\end{align*}
where we use rotational invariance in the first equality and \eqref{thmGeo722} with $q=k$ and $r=0$ to evaluate the double integral. Therefore, 
\[
I(L)=    \frac{b_{k,q}}{b_{n,q}}\kappa_n^q
\int_{G(L,q)} g(M) [M,L_0]^{n-k}\nu_q^{L}(d M),
\]
can be inserted into \eqref{eq:zrech}
to arrive at
\begin{align*}
     \int_{G(n,q)} g(L) \, \nu_q(dL)
    =
    \frac{b_{n-r,q}}{b_{k-r,q}} \frac{b_{k,q}}{b_{n,q}} \int_{G(L_0,k)} 
     \int_{G(L,q)} g(M) 
     [M,L_0]^{n-k}
     \nu_q^{L}(d M)
     \nu_k^{L_0}(dL).
\end{align*}
Letting $g(L)= \int_{L^\perp} f(L+x) \, \lambda_L(dx)$ we arrive at
\begin{align*}
  \frac{b_{k-r,q}}{b_{n-r,q}} \frac{b_{n,q}}{b_{{k,q}}}\int_{A(n,q)} f(E) \,dE
    &=
   \int_{G(L_0,k)} 
     \int_{G(L,q)}  \int_{M^\perp} f(M+x) \,
     [M,L_0]^{n-k}
     \\ \,& \qquad \times 
     \lambda_{M^\perp}(dx)
    \,\nu_q^{L}(d M)
    \,\nu_k^{L_0}(dL).
\end{align*}
The  use of \eqref{eqAffineBPhelp} concludes the proof.
\phantom\qedhere
\end{proof}
Notice how the proof proceeds along the same lines as the proof of Theorem \ref{thm:BP}: We first transfer integration over subspaces to integration over tuples of vectors and then apply \eqref{thmGeo722}. It is also worth mentioning that we now allow for a larger range for the dimension $q$ than in Theorem \ref{thm:BP}, as $q=k-r$ is now allowed. 
The new normalizing constant closely resembles $\alpha_{n,k,q,r}$ and can be reduced into a similar structure. In fact, using $b_{a,q}= \frac{\omega_{q+1}}{\omega_{a-q}} b_{a,q+1}$ valid when $q+1\leq a$ yields
\begin{equation}\label{eq:ReductionConstantAppendix}
   \frac{b_{k-r,q}}{b_{n-r,q}} \frac{b_{n,q}}{b_{{k,q}}}
   =
   \alpha_{n,k,q,r} \frac{\omega_{n-r-q}}{\omega_{k-r-q}}
\end{equation}
under the assumptions in Theorem \ref{thm:verticalBP}. 

\begin{proof}[Theorem \ref{thmGeoAffineRotationalCroften}]
With the assumptions of Theorem \ref{thmGeoAffineRotationalCroften}, put $f(E)=V_{q-j}(K\cap E)$ in the affine Blaschke--Petkantschin formula (Theorem \ref{thm:verticalBP}) and observe that the right-hand side is a classical Crofton integral that can be evaluated using \eqref{thmGeo511Crofton}. The constant is reduced using \eqref{eq:ReductionConstantAppendix}. 
\phantom\qedhere
\end{proof}

We conclude this appendix with a result corresponding to Theorem \ref{Thm:Uniqueness}.

	\begin{thm}\label{Thm:UniquenessAffine}
		Let the setting be as in Theorem \ref{thmGeoAffineRotationalCroften}. Then, 
		$\tilde{\varphi}_{L,q}^{L_0}(L \cap K)$ is independent of 
		$q\in  \{j,\ldots, k-r\}$.
	\end{thm}
	 \begin{proof}
 The proof of this theorem is a slight modification of the proof leading to Theorem \ref{Thm:Uniqueness}. We will therefore refer to this proof for expanding comments.

 Let the assumptions of Theorem \ref{thmGeoAffineRotationalCroften} be satisfied. If $\{j, \ldots ,k-r\}$ is a singleton, the claim is trivial. Thus, we may assume $j <k-r$ and 
it is enough to show 
\begin{equation}
		\tilde\varphi_{L+x,q+1}^{L_0}(\cdot)=\tilde\varphi_{L+x,q}^{L_0}( \cdot)
	\label{eq:whatwewantaffine}
\end{equation}	
for every fixed $q\ge j$ such that $q +1 \leq k-r$. 
Copying the arguments leading to \eqref{eq:q+1q} we conclude that
		\begin{equation}
	\tilde\varphi_{L+x,q+1}^{L_0}(K \cap (L+x))
			=
			\tilde c_1 \int_{A(L+x,q)} V_{q-j}(K \cap E) \tilde h(E) \,\mu_q^L(dE),
			\label{eq:q+1qAffine}
		\end{equation}
		with $\tilde c_1=  c_0(q)  \frac{\omega_{k-q}}{\omega_{n-q}} $, where now 
		\begin{equation*}
			\tilde h(E) =\int_{A(E,q+1)} [\lin E_1,L_0]^{n-k} \,\mu_q^L(dE_1)
		\end{equation*}
  is to be understood in the $k$-dimensional affine space $L+x$. In analogy to the proof of Theorem \ref{Thm:Uniqueness}, the flat $L+x$ is  identified with $\R^k$ to  evaluate $\tilde h$. 
 Writing $E= M+z$ with $M \in G(k,q)$ and $z \in M^\perp$, \cite[eq.~(13.14)]{GulBog} gives
\begin{align*}
  \tilde h(M+z)
	&=\int_{G(M,q+1)}
	[\lin (L_1+z),L_0]^{n-k} \,\nu_{q+1}^M(dL_1)
	\\&=
	\omega_{k-q}^{-1}
	\int_{S^{k-1} \cap M^\perp }
	[M+ \myspan u,L_0]^{n-k} \,\mathcal{H}^{k-q-1}(du) 
	\\&=\omega_{k-q}^{-1} [M,L_0]^{n-k} 
	\int_{S^{k-1} \cap M^\perp }
	\| p(u |(L_0+M)^\perp )  \|^{n-k} \, \mathcal{H}^{k-q-1}(du),
\end{align*}
where we used  
   \begin{equation}
       [M+\myspan u,L_0] = [M,L_0] \,\|p(u|(M+L_0)^\perp \|
       \label{eqNewF}
   \end{equation} at the last equality sign. 
   Equation \eqref{eqNewF} follows directly from \eqref{eq:Gul7.9}, and is the main difference compared to the proof of Theorem \ref{Thm:Uniqueness} where Lemma \ref{lemD(E,L)} was used instead. 
   As $[M,L_0]=[\lin (E),L_0]$,  we can conclude \eqref{eq:whatwewantaffine} and thereby end the proof by proving that 
	  		\begin{equation}
	  	\tilde I(M):=  \int_{S^{k-1} \cap M^\perp }
	  	\| p(u  |(L_0+M)^\perp )  \|^{n-k} \,\mathcal{H}^{k-q-1}(du)
	  	\label{eq:IMAffine}
	  \end{equation}
   is constant except on a suitable null-set. Applying \eqref{eq:brugesIAppendix} now with $d=k-q$, $j=n-k$ and $p = k-r-q$ gives
   \begin{equation*}
       \tilde I(M)= \frac{\omega_{n-q}}{\omega_{n-r-q}} \omega_{k-r-q},
   \end{equation*}
   except on the set where $\dim( M+L_0) < r+q$,  which is the desired null-set. 
Collecting all the constants, we conclude   
   \begin{equation*}
       \tilde \varphi_{L+x,q+1}
       = c_0(q) \frac{\omega_{k-q}}{\omega_{n-q}}
       \frac{1}{\omega_{k-q}}
        \frac{\omega_{n-q}}{\omega_{n-r-q}} \omega_{k-r-q}
       d(q)^{-1}
  \tilde \varphi_{L+x,q}  =
    \tilde \varphi_{L+x,q},
   \end{equation*}
   ending the proof.
   \phantom\qedhere
	\end{proof}
\end{document}